\documentclass[a4paper,11pt]{amsart}
\usepackage{amssymb}
\usepackage{amsmath}
\usepackage{amsthm}
\usepackage{amscd}
\usepackage{verbatim}
\usepackage{graphicx}
\usepackage{url}
\usepackage[all]{xy}

\usepackage[top=2.7cm, bottom=3.0cm, left=3.0cm, right=3.0cm]{geometry}
\DeclareFontEncoding{OT2}{}{} % to enable usage of cyrillic fonts
\DeclareTextFontCommand{\textcyr}{\fontencoding{OT2}
    \fontfamily{wncyr}\fontseries{m}\fontshape{n}\selectfont}

\theoremstyle{plain}

\newtheorem{theorem}{Theorem}[section]

\newtheorem{lemma}[theorem]{Lemma}
\newtheorem{corollary}[theorem]{Corollary}

\theoremstyle{definition}

\newtheorem{definition}[theorem]{Definition}

\newtheorem{construction}[theorem]{Construction}

\title[Nonabelian $H^2$]
{Nonabelian $H^2$ with coefficients in a group\\ and with coefficients in a crossed module}

\author{Mikhail Borovoi}

\begin{document}

%%%%%%%%%%%%%%%%%%%%%%%%%%%%%%%%%%%%%%%%%%%%%%%%%%
%%%%%%%%%%%%%%%%%%%%%%%%%%%%%%%%%%%%%%%%%%%%%%%%%%

\date{\today}

\begin{abstract}
In this  note, following Dedecker and Debremaeker,
we extend the  group cohomology exact sequence for $H^1$ with coefficients in groups, using $H^2$ with coefficients in crossed modules.
\end{abstract}

\maketitle

%\tableofcontents

\def\hs{\kern 0.8pt}
\def\hl{\kern -2.0pt}
\def\hlb{\kern -0.3pt }

\def\G{{\Gamma}}
\newcommand{\labelto}[1]{\xrightarrow{\makebox[1.5em]{\scriptsize ${#1}$}}}
\newcommand{\labelt}[1]{\xrightarrow{\makebox[0.8em]{\scriptsize ${#1}$}}}
\def\onto{\twoheadrightarrow}
\def\into{\hookrightarrow}
\newcommand{\isoto}{\overset{\sim}{\to}}
\def\Maps{{\mathrm{Maps}}}
\def\H{{\mathbf{H}}}
\def\vk{\varkappa}
\def\Aut{\mathrm{Aut\,}}
\def\Inn{\mathrm{Inn\,}}
\def\inn{\mathrm{inn}}
\def\Out{\mathrm{Out\,}}
\def\lra{\longrightarrow}

\def\upsig{\hs^\sigma\!}
\def\uptau{\hs^\tau\!}
\def\ups{\upsilon}
\def\sig{\sigma}
\def\upg{\hs^g\!}

\def\rel{\mathrm{\ rel\ }}

\def\im{\mathrm{im\,}}

%%%%%%%%%%%%%%%%%%%%%%%%%%%%%%%%%%%%%%%%%%%%%%%%%%%%%%%%%%%%%%%%%%%%%%%%%

%\setcounter{section}{-1}

Let $\Gamma$ be a fixed group. In this note we consider groups with $\Gamma$-action and crossed modules with $\Gamma$-action.
Following Dedecker and Debremaeker, we extend the cohomology exact sequence for $H^1$ with coefficients in groups,
using $H^2$ with coefficients in crossed modules.
The obtained exact sequence seems to be essentially equivalent to the sequence of Springer \cite[Props.\, 1.27, 1.28, 1.29]{Springer66},
but looks nicer from our point of view.
The difference of our approach to $H^2$ of crossed modules here from the approach of \cite{Breen90} and  \cite{Borovoi98} is that here,
following Dedecker and Debremaeker, we equip $H^2$ of a crossed module with a large set of {\em neutral elements}.
According to Debremaeker \cite{Debremaeker-thesis}, all our results are valid in the more general context of groups and crossed modules of a topos.
We claim no originality.

\section{Second cohomology with coefficients in a crossed module}

Let $(A\labelt{\rho} G)$ be a left crossed module with a $\G$-action (see below).
The second nonabelian cohomology with coefficients in a crossed module was considered in
\cite{Dedecker64}, \cite{Debremaeker-thesis}, \cite{Debremaeker77}, \cite{Breen90}, \cite{Borovoi98}, \cite{Noohi11}.
We define $H^2(A\to G)$ in terms of cocycles (note that in \cite{Borovoi98} this set was denoted by $H^1(\G, A\to G)$,
while in \cite{Breen90} the corresponding set in a more general setting  was denoted by $H^1(A\to G)$\,).
It is important that the set $H^2(A\to G)$ has a distinguished element
(the {\em unit element}) and a distinguished subset of {\em neutral elements}.

\begin{definition}
A (left) crossed module is a homomorphism of groups $\rho\colon A\to G$ together with a left action $G\times A\to A$ of $G$ on $A$,
denoted $(g,a)\mapsto \hs ^a\hl g$, such that
\begin{align*}
a a' a^{-1}&=\hs^{\rho(a)} a',\\
\rho(\hs^g\hlb a&)=g\cdot \rho(a)\cdot g^{-1}
\end{align*}
for all $a,a'\in A,\ g\in G$.
\end{definition}

For examples of crossed modules see e.g. \cite[Examples 3.2.2]{Borovoi98}.
Note that for any group $A$ we have crossed modules $A\to\Aut A$ and $A\to\Inn A$.

We say that our fixed  group $\Gamma$ acts on a crossed module $(A\to G)$ if $\G$ acts on $A$ and $G$ so that
\[\rho(\hs^\sigma\hl a)=\hs^\sigma(\rho(a)),\qquad ^\sigma(\hs^g\hlb a)=\hs^{^\sigma\hl g}\hlb(\hs^\sigma\hl a) \]
for all $a\in A,\ g\in G,\ \sigma\in \G$.

Let $Z^2(\G,A\to G)$ denote the set of pairs $(u,\psi)$, where
\[u\colon \G\times\G\to A \quad\text{and}\quad \psi\colon\G\to G\]
are maps satisfying the cocycle conditions of \cite[(3.3.2.1-2)]{Borovoi98}:
\begin{align*}
& u_{\sigma,\tau\upsilon}\cdot \psi_\sigma(\hs^\sigma\! u_{\tau,\upsilon})= u_{\sigma\tau,\upsilon}\cdot u_{\sigma,\tau}\\
&\psi_{\sigma\tau}=\rho(u_{\sigma,\tau})\cdot\psi_\sigma\cdot\hs^\sigma\! \psi_\tau
\end{align*}
for all $\sigma,\tau,\upsilon\in\G$.

\begin{construction}\label{cons:action}
We define a left action of the group $\Maps(\G,A)$ on $Z^2(\G,A\to G)$ as follows.
If
\[w\in\Maps(\G,A),\ (u,\psi)\in Z^2(\G, A\to G),\]
 then we set
\[w*(u,\psi)=(u',\psi'),\]
where
\begin{align*}
& u'_{\sigma,\tau}=w_{\sigma\tau}\cdot u_{\sigma,\tau}\cdot\psi_\sigma(\hs^\sigma\!w_\tau)^{-1}\cdot w_\sigma^{-1},\\
&\psi'_\sigma=\rho(w_\sigma)\cdot\psi_\sigma
\end{align*}
for all $\sigma,\tau\in\G$.
One checks that $(u',\psi')\in Z^2(\G, A\to G)$, see Appendix \ref{s:A} below.

We define a left  action of $G$ on $Z^2(\G, A\to G)$ as follows.
If
\[g\in G,\ (u,\psi)\in Z^2(\G,A\to G),\]
 then we set
\[g\star(u,\psi)=(u'',\psi''),\]
where
\begin{align*}
&u''_{\sigma,\tau}=\hs^g \hlb u_{\sigma,\tau}\\
&\psi''_\sigma=g\cdot\psi_\sigma\cdot\hs^\sigma\!\hl g^{-1}
\end{align*}
for all $\sigma,\tau\in G$.
One checks that   $(u'',\psi'')\in Z^2(\G, A\to G)$, see Appendix \ref{s:A} below.
\end{construction}

The group $G$ acts on the left on the group $\Maps(\G,A)$ by
\[g\star w=w',\quad \text{where } w'_\sigma=\hs^g\hlb w_\sigma\]
for $g\in G$, $w\in\Maps(\G,A)$, and  $\sigma\in\G$.
We consider the semi-direct product
\[C^1(\G,A\to G):=\Maps(\Gamma,A)\rtimes G.\]
Then the group $C^1(\G,A\to G)$ acts on the left on the set $Z^2(\G, A\to G)$.
Following Dedecker \cite{Dedecker64} and \cite{Dedecker69}, we define the {\em thick} cohomology set and the the {\em thin} cohomology set.

\begin{definition}
The {\em thick} cohomology set is
\[\H^2(A\to G):=Z^2(\G, A\to G)/\Maps(\G,A).\]
\end{definition}

\begin{definition}
The {\em thin} cohomology set is
\[H^2(A\to G):=Z^2(\G, A\to G)/C^1(\G,A\to G)=\H^2(A\to G)/G.\]
\end{definition}

We have a canonical surjective map
\begin{equation}\label{e:vk}
\vk\colon\H^2(A\to G)\to H^2(A\to G).
\end{equation}

\begin{definition} The {\em unit cocycle} in $Z^2(\G, A\to G)$ is the cocycle $(1_A,1_G)$.
The {\em unit classes} in $\H^2(A\to G)$ and $H^2(A\to G)$ are the images of the unit cocycle.
\end{definition}

\begin{definition} A {\em neutral cocycle} in $Z^2(\G, A\to G)$ is a cocycle  of the form $(1_A,\psi)$.
The {\em neutral classes} in $\H^2(A\to G)$ and $H^2(A\to G)$ are the images of the neutral cocycles.
\end{definition}

Thus the set $H^2(A\to G)$ contains the distinguished subset $H^2(A\to G)'$ of neutral elements.
This subset $H^2(A\to G)'$ contains the  unit element 1.

\section{Second cohomology with coefficients in a group}
Let $A$ be a $\G$-group. The $\Gamma$-action defines a homomorphism
\[f_A\colon \G\to\Aut A,\ (f_A)_\sigma(a)=\hs^\sigma\! a,\]
and thus it defines a $\G$-kernel ($\G$-band, $\G$-lien)
\[\kappa_A\colon \G\labelto{f_A} \Aut A\labelt{}\Out A,\]
where $\Out A:=\Aut A/\hs\Inn A$.
We write $H^2(A)$ for $H^2(\G, A,\kappa_A)$.
The second nonabelian cohomology set $H^2(A)$ was defined by Springer \cite{Springer66} and Giraud \cite{Giraud71},
see also \cite{Borovoi93}, \cite{FSS98}, \cite{Florence04}, and \cite{LA15}.
By definition \cite[Section 1.5]{Borovoi93}, the set of 2-cocycles $Z^2(\G,A,\kappa_A)$ is the set of pairs $(u,f)$,
where $u\in\Maps(\G\times \G)\to A$ and  $f\in \Maps(\G,\Aut A)$,
satisfying the 2-cocycle conditions
\begin{align*}
& f_{\sigma\tau}=\inn(u_{\sigma,\tau})\circ f_\sigma\circ f_\tau\\
& u_{\sigma,\tau\upsilon}\cdot f_\sigma(u_{\tau,\upsilon})=u_{\sigma\tau,\upsilon}\cdot u_{\sigma,\tau}\\
&f_\sigma=\psi_\sigma\circ (f_A)_\sigma\quad\text{for some } \psi_\sigma\in \Inn A
\end{align*}
for all $\sigma,\tau,\upsilon\in\G$.
The group $\Maps(\G,A)$ acts on the left on $Z^2(\G,A,\kappa_A)$ as follows.
If
\[w\in\Maps(\G,A),\ (u,f)\in Z^2(\G, A,\kappa_A),\]
then
\[w*(u,f)=(u',f'),\]
where
\begin{align*}
& u'_{\sigma,\tau}=w_{\sigma\tau}\cdot u_{\sigma,\tau}\cdot f_\sigma(w_\tau)^{-1}\cdot w_\sigma^{-1},\\
& f'_\sigma=\inn(w_\sigma)\circ f_\sigma
\end{align*}
for all $\sigma,\tau\in \G$.

\begin{definition}  $H^2(A)=Z^2(\G,A,\kappa_A)/\Maps(\G,A)$.  \end{definition}

By a {\em neutral} cocycle in $Z^2(\G,,A,\kappa_A)$ we mean a cocycle
of the form $(1_A,f)$, and by the {\em unit} cocycle we mean $(1_A, f_A)$.
A {\em neutral class} in $H^2(A)$ is the class of a neutral cocycle,
and the {\em unit class} 1 in $H^2(A)$ is the class of the unit cocycle.
We obtain a distinguished subset $H^2(A)'\subset H^2(A)$
consisting of the neutral elements, and $H^2(A)'$ contains the unit element 1.

Note that a 2-cocycle $(u,f)\in Z^2(\G,A,\kappa_A)$ defines a map  $\psi\colon \G\to\Inn A$
by $\psi_\sigma=f_\sigma\circ(f_A)_\sigma^{-1}$.
One checks immediately that $(u,\psi)\in Z^2(\G, \Inn A)$ and that the bijection
\[Z^2(\G,A,\kappa_A)\labelt\sim Z^2(\G,A\to\Inn A), \quad(u,f)\mapsto (u,\psi), \]
induces a canonical bijection $H^2(A)\labelt\sim \H^2(A\to\Inn A)$.
This bijection induces a bijection $H^2(A)'\labelt\sim \H^2(A\to\Inn A)'$ on the set of neutral elements
and takes the unit element of $H^2(A)$ to the unit element of $\H^2(A\to\Inn A)$.
We obtain a canonical surjective map
\begin{equation}\label{e:lambda}
\lambda_A\colon H^2(A)\labelt\sim \H^2(A\to\Inn A)\labelt{\vk} H^2(A\to\Inn A).
\end{equation}

\begin{theorem}[Debremaeker {\cite[Ch.\,V, Thm.~3, p.\,112]{Debremaeker-thesis}}]
\label{t:Deb}
The canonical surjection \eqref{e:lambda} is a bijection.
\end{theorem}
\begin{proof}[First proof]
Let $Z_A$ denote the center of $A$, which is is a $\G$-group.
Then the group of 2-cocycles $Z^2(\G,Z_A)$ acts on the left on the set $Z^2(\G,A,\kappa_A)$ as follows:
\[\text{if } z\in Z^2(\G,Z_A),\ (u,f)\in Z^2(\G,A,\kappa_A),\text{ then } z*(u,f)=(zu,f).\]
This action induces an action of $H^2(Z)$ on $H^2(A)$, which is simply transitive,
see \cite[IV-Thm.\,8.8]{MacLane63} or \cite[Prop.\,1.17]{Springer66}.

On the other hand, the group $Z^2(\G,Z_A)$ acts on the left on the set $Z^2(\G,A\to \Inn A)$ as follows:
\[\text{if } z\in Z^2(\G,Z_A),\ (u,\psi)\in Z^2(\G,A\to \Inn A),\text{ then } z*(u,\psi)=(zu,\psi).\]
This map induces an action of $H^2(Z_A)$ on $H^2(A\to \Inn A)$ and, by the action of $H^2(Z_A)$
on the unit element $1\in H^2(A\to \Inn A)$,
it induces a map
\[\mu\colon H^2(Z_A)\to H^2(A\to \Inn A),\]
which can be factored as
\[\mu\colon H^2(Z_A)\isoto H^2(Z_A\to 1)\labelt{\iota_*} H^2(A\to\Inn A),\]
where the map $\iota_*$ is induced by the embedding of crossed modules
\[ \iota\colon (Z\to 1)\into(A\to\Inn A).\]
Since the embedding $\iota$ is a quasi-isomorphism of crossed modules, the map $\iota_*$ is bijective
(\hs see \cite[Thm.~3.3]{Borovoi92} or \cite[Prop.~5.6]{Noohi11}\hs),
hence the map $\mu$ is bijective and the action of  $H^2(Z_A)$ on $H^2(A\to \Inn A)$ is simply transitive.

Now, since the map \eqref{e:lambda} is $H^2(Z_A)$-equivariant, we conclude that it is bijective, as required.

\noindent
{\em Second proof} (similar to \cite[Proof of Prop. 1.19]{Springer66}).
We wish to prove that the surjective map
\[\vk\colon  \H^2(A\to\Inn A)\to H^2(A\to \Inn A)\]
is bijective.
It suffices to show that $\Inn A$ acts on $\H^2(A\to\Inn A)$ trivially.

Let
\[g\in\Inn A,\ g=\inn(b),\ b\in A.\]
One can check  that the cocycles  $g*(u,\psi)$ and $(u,\psi)$ give the same class in $\H^2(A\to\Inn A)$, namely, that
\[g*(u,\psi)=w*(u,\psi),\]
where
\[ w\in\Maps(\G,A),\ w_\sigma=b\cdot\psi_\sigma(\hs^\sigma\! b)^{-1}\ \text{ for } \sigma\in\G,\]
see Appendix \ref{s:B} below.
This completes the second proof.
\end{proof}

\section{Cohomology exact sequence}
Let
\begin{equation}\label{e:ABC}
1\to A\labelt{i} B\labelt{j} C\to 1
\end{equation}
be a short exact sequence of $\G$-groups.
We construct a connecting map
\[\Delta\colon H^1(C)\to H^2(A\to\Inn B)\quad \text{(sic!).} \]

Let $c\in Z^1(\G,C)\subset\Maps(\G,C)$. We lift $c$ to some map $b\colon \G\to B$ and define
\begin{align*}
&u_{\sigma,\tau}=b_{\sigma\tau}\cdot\hs^\sigma\hlb b_\tau^{-1}\cdot b_\sigma^{-1}\in A\\
&\psi_\sigma=\inn(b_\sigma)\in\Inn B
\end{align*}
for $\sigma,\tau\in \G$.
We set $\Delta([c])=[u,\psi]$, where $[c]$ denotes the class of $c$ in $H^1(C)$
and $[u,\psi]$ denotes the class of $(u,\psi)\in Z^2(\G,A\to\Inn B)$ in $H^2(A\to\Inn B)$.
One checks check that $(u,\psi)\in Z^2(\G,A\to\Inn B)$ and that the map $\Delta$ is well defined,
see Appendix \ref{s:C} below.

Consider the morphisms of crossed modules
\[\xymatrix@C=19pt{ (A\to\Inn B)\ar[r]^{i_*} &(B\to\Inn B)\ar[r]^{j_*} &(C\to \Inn C)  }\]
and the sequence
\begin{equation}\label{e:DD}
\xymatrix@C=19pt{
H^1(B)\ar[r]^{j_*} &H^1(C)\ar[r]^-{\Delta} &H^2(A\to\Inn B)\ar[r]^{i_*} &H^2(B\to\Inn B)\ar[r]^{j_*} &H^2(C\to \Inn C)
}
\end{equation}

\begin{theorem}[Dedecker {\cite[Thm.\,2.2]{Dedecker69}} and Debremaeker {\cite[Ch.\,IV, Thm.\,2.1.7, p.\,103]{Debremaeker-thesis}}]
\label{t:DD}
For an exact sequence of $\G$-groups \eqref{e:ABC}, the sequence  \eqref{e:DD} is exact in the following sense:
\begin{enumerate}
\item[(i)] an element of $H^1(C)$ is contained in the image of $H^1(B)$
if and only if its image in $H^2(A\to\Inn B)$ is {\em neutral};
\item[(ii)] an element of $H^2(A\to \Inn B)$ is contained in the image  of $H^1(C)$
if and only if its image in $H^2(B\to\Inn B)$ is {\em the unit element};
\item[(iii)] an element of $H^2(B\to\Inn B)$ is contained in the image of $H^2(A\to\Inn B)$
if and only if its image in $H^2(C\to\Inn C)$ is {\em neutral}.
\end{enumerate}
\end{theorem}

\begin{proof}
See Appendix \ref{s:C} below.
\end{proof}

\section{A version of Theorem \ref{t:DD}.}

We write
\[ G= (\Inn B)|_A:=\{\inn(b)|_A\ |\ b\in B\},\]
the group of restrictions to $A$ of the inner automorphisms of $B$.
 Then $G\subset \Aut A$.
We have an epimorphism $\Inn B \to G$ and a morphism of crossed modules
\[\pi\colon (A\to\Inn B)\ \to\ (A\to G).\]

\begin{lemma}
For $(u,\psi)\in Z^2(\Gamma, A\to\Inn B)$, its class $[u,\psi]\in H^2(A\to\Inn B)$ is neutral if and only if
$\pi_*([u,\psi])\in H^2(A\to G)$ is neutral.
\end{lemma}
\begin{proof} Easy. \end{proof}

\begin{corollary}\label{c:DD}
For an exact sequence of $\G$-groups \eqref{e:ABC}, the sequence
\begin{equation}
\xymatrix@C=30pt{
H^1(B)\ar[r]^{j_*} &H^1(C)\ar[r]^-{\Delta\circ\pi_*} &H^2(A\to G)
}
\end{equation}
is exact in the following sense: a cohomology class $c\in H^1(C)$ comes from $H^1(B)$ if and only if its image in $H^2(A\to G)$ is neutral.
\end{corollary}

\section{Example}
We compute the map $\Delta\circ \pi_*\colon H^1(C)\to H^2(A\to G)$ in the case when $A$ is {\em abelian}.
Here $G=(\Inn B)|_A$.

Let $(u,\psi)\in Z^2(\G, A\to G)$, then
\begin{align*}
& u_{\sigma,\tau\upsilon}\cdot \psi_\sigma(\hs^\sigma\! u_{\tau,\upsilon})= u_{\sigma\tau,\upsilon}\cdot u_{\sigma,\tau}\\
&\psi_{\sigma\tau}=\inn(u_{\sigma,\tau})\cdot\psi_\sigma\cdot\hs^\sigma\! \psi_\tau\,.
\end{align*}
Since $A$ is abelian, the homomorphism $A\to G$ is trivial, hence $\psi$ is a 1-cocycle, $\psi\in Z^1(\G,G)$.
Moreover, $u$ is a 2-cocycle, $u\in Z^2(\G,\hs_\psi A)$.
One checks immediately that the map
\[  Z^2(\G, A\to G)\to Z^1(\G,G),\quad (u,\psi)\mapsto \psi \]
induces a map
\[\zeta\colon H^2( A\to G)\to H^1(G),\quad [u,\psi)]\mapsto [\psi]. \]
Moreover, for given $\psi\in Z^1(\G,G)$, we have a bijection
\[ \lambda_\psi\colon H^2(\hs_\psi A)\isoto \zeta^{-1}([\psi]) ,\quad  [u] \mapsto[u,\psi],\]
and $\lambda_\psi([u])$ is neutral in $H^2(A\to G)$ if and only if $[u]=0\in H^2(\hs_\psi A)$.

Since $A$ is abelian, $C$ acts on $A$, and we obtain a surjective homomorphism $p\colon C\to G$.
Let $c\in Z^1(\G,C)$, $\psi=p_*(c)\in Z^1(\G,G)$, then we write $_c A$ for $_\psi A$.

Let us lift $c\colon \G\to C$ to some map $b\colon \G\to B$ and set
\[ u_{\sigma,\tau}= b_{\sigma\tau}\cdot\hs^\sigma\hlb b_\tau^{-1}\cdot b_\sigma^{-1}\,, \]
then $u\in Z^2(\G,\hs_c A)$.
Set
\[ \Delta_S(c)=[u]\in H^2(\hs_c A).\]
A simple computation shows that the image of $[c]$ in $H^2(A\to G)$ is
\[ \lambda_\psi( \Delta_S(c))\in \zeta^{-1}([\psi])\subset H^2(A\to G), \]
where $\psi=p_*(c)\in Z^1(\G, G)$, \ $\Delta_S(c)\in H^2(\hs_c A)$.
This image is a neutral class in $H^2(A\to G)$ if and only if $\Delta_S(c)=0$.

Applying Corollary \ref{c:DD},
we recover a result of Serre \cite[I.5.6, Prop.\,41]{Serre}: a cohomology class  $[c]\in H^1(C)$ comes from $H^1(B)$ if and only if
\[\Delta_S(c)=0\in H^2(\hs_c A).\]

\appendix

\section{Checks in Construction \ref{cons:action}}
\label{s:A}

We define a left action of the group $\Maps(\G,A)$ on $Z^2(\G,A\to G)$ as follows.
If
\[w\in\Maps(\G,A),\ (u,\psi)\in Z^2(\G, A\to G),\]
 then we set
\[w*(u,\psi)=(u',\psi'),\]
where
\begin{align*}
& u'_{\sigma,\tau}=w_{\sigma\tau}\cdot u_{\sigma,\tau}\cdot\psi_\sigma(\hs^\sigma\!w_\tau)^{-1}\cdot w_\sigma^{-1},\\
&\psi'_\sigma=\rho(w_\sigma)\cdot\psi_\sigma
\end{align*}
for all $\sigma,\tau\in\G$.

We check that $(u',\psi')\in Z^2(\G, A\to G)$. We have
\begin{align*}
&u'_{\sig,\tau\ups}\cdot\psi'_\sig(\upsig u'_{\tau,\ups})=\\
& = w_{\sig\tau\ups}\cdot u_{\sig,\tau\ups}\cdot \psi_\sig(\upsig w_{\tau\ups})^{-1}\cdot w_\sig^{-1}\cdot
w_\sig\cdot\psi_\sig(\upsig(w_{\tau\ups}\cdot u_{\tau,\ups}\cdot \psi_\tau(\uptau u_\ups)^{-1}\cdot w_\tau^{-1}))\cdot w_\sig^{-1}\\
& =  w_{\sig\tau\ups}\cdot u_{\sig,\tau\ups}\cdot
\psi_\sig(\upsig u_{\tau,\ups})^{-1}\cdot\psi_\sig(\upsig(\psi_\tau(\uptau w_\ups)))^{-1}\cdot
\psi_\sig(\upsig w_\tau)^{-1}\cdot w_\sig^{-1}\,,
\end{align*}
\begin{align*}
 u'_{\sig\tau,\ups}\cdot u'_{\sig,\tau}
& = w_{\sig\tau\ups}\cdot u_{\sigma\tau,\ups}\cdot \psi_{\sigma\tau}(\hs^{\sigma\tau}\! w_\ups)^{-1}\cdot w_{\sigma\tau}^{-1}\cdot
w_{\sig\tau}\cdot u_{\sig,\tau}\cdot\psi_\sig(\upsig w_\tau)^{-1}\cdot w_\sig^{-1}\\
& =  w_{\sig\tau\ups}\cdot u_{\sigma\tau,\ups}\cdot u_{\sig,\tau}\cdot
\psi_\sig(\upsig(\psi_\tau(\uptau w_\ups)))^{-1}\cdot \psi_\sig(\upsig w_\tau)^{-1}\cdot w_\sig^{-1}\\
& = w_{\sig\tau\ups}\cdot u_{\sig,\tau\ups}\cdot \psi_\sig(\upsig u_{\tau,\ups})^{-1}
\cdot\psi_\sig(\upsig(\psi_\tau(\uptau w_\ups)))^{-1}\cdot
\psi_\sig(\upsig w_\tau)^{-1}\cdot w_\sig^{-1}\\
& =u'_{\sig,\tau\ups}\cdot\psi'_\sig(\upsig u'_{\tau,\ups}),
\end{align*}
as required.
We have
\begin{align*}
\rho(u'_{\sig,\tau})\cdot\psi'_\sig\cdot\upsig\psi'_\tau
& = \rho(w_{\sig\tau}\cdot u_{\sig,\tau}\cdot\psi_\sig(\upsig w_\tau)^{-1}\cdot w_\sig^{-1})\cdot
\rho(w_\sig)\cdot\psi_\sig\cdot\upsig(\rho(w_\tau)\psi_\tau)\\
& =\rho( w_{\sig\tau}\, u_{\sig,\tau})\cdot\rho(\psi_\sig(\upsig w_\tau))^{-1}
\cdot\psi_\sig\cdot\rho(\upsig w_\tau)\cdot\upsig\psi_\tau\\
& =\rho(w_{\sig\tau})\cdot\rho(u_{\sig,\tau})\cdot\psi_\sig\cdot\upsig\psi_\tau=\rho(w_{\sig\tau})\cdot\psi_{\sig\tau}=\psi'_{\sig\tau}\,,
\end{align*}
as required.

We define a left  action of $G$ on $Z^2(\G, A\to G)$ as follows.
If
\[g\in G,\ (u,\psi)\in Z^2(\G,A\to G),\]
 then we set
\[g\star(u,\psi)=(u'',\psi''),\]
where
\begin{align*}
&u''_{\sigma,\tau}=\hs^g \hlb u_{\sigma,\tau}\\
&\psi''_\sigma=g\cdot\psi_\sigma\cdot\hs^\sigma\!\hl g^{-1}
\end{align*}
for all $\sigma,\tau\in G$.

We check that   $(u'',\psi'')\in Z^2(\G, A\to G)$. We have
\begin{align*}
u''_{\sigma,\tau\ups}\cdot\psi''_\sigma(\upsig u''_{\tau,\ups})
&=\upg u_{\sigma,\tau\ups}\cdot(g\hs\psi_\sigma\hs\upsig g^{-1})(\upsig(\upg u_{\tau,\ups})) =\upg u_{\sigma,\tau\ups}\cdot\upg\psi_\sigma(\upsig u_{\tau,\ups})\\
&=\upg(u_{\sigma,\tau\ups}\cdot\psi_\sigma(\upsig u_{\tau,\ups}))=\upg(u_{\sigma\tau,\ups}\cdot u_{\sigma,\tau})\\
&=\upg u_{\sigma\tau,\ups}\cdot\upg u_{\sigma,\tau}=u''_{\sigma\tau,\ups}\cdot u''_{\sigma,\tau}\,,
\end{align*}
\begin{align*}
\rho(u''_{\sigma,\tau})\cdot\psi''_\sigma\cdot\upsig\psi''_\tau
&= \rho(\upg u_{\sigma,\tau})\cdot g\hs\psi_\sigma\hs\upsig g^{-1}\cdot\upsig(g\hs\psi_\tau\hs\uptau g^{-1})\\
&=g\hs\rho(u_{\sigma,\tau})\hs g^{-1}\cdot g\hs \psi_\sigma\hs\upsig g^{-1}\cdot \upsig g \hs\upsig\psi_\tau\hs\hs^{\sigma\tau}\! g^{-1}\\
&=g\cdot\rho(u_{\sigma,\tau})\cdot\psi_\sigma\cdot\upsig\psi_\tau\cdot
\hs^{\sigma\tau}\!g^{-1}=g\cdot\psi_{\sigma\tau}\cdot\hs^{\sigma\tau}\!g^{-1}=\psi''_{\sigma\tau}\,,
\end{align*}
as required.

\section{Checks in the proof of Theorem \ref{t:Deb}}
\label{s:B}

Let
\[g\in\Inn A,\ g=\inn(b),\ b\in A.\]
We check  that the cocycles  $g*(u,\psi)$ and $(u,\psi)$ give the same class in $\H^2(A\to\Inn A)$, namely, that
\[g*(u,\psi)=w*(u,\psi),\]
where
\[ w\in\Maps(\G,A),\ w_\sigma=b\cdot\psi_\sigma(\hs^\sigma\! b)^{-1}\ \text{ for } \sigma\in\G.\]

Indeed, write
\[  g*(u,\psi)=(u'',\psi''),\quad w*(u,\psi)=(u',\psi'),\]
then
\begin{align*}
&u''_{\sigma,\tau} = b\hs u_{\sigma,\tau}\hs b^{-1}, \quad
                                         \psi''_\sigma=\inn(b)\circ\psi_\sigma\circ\inn(\hs^\sigma\! b)^{-1},\\
&u'_{\sigma,\tau}= w_{\sigma\tau}\cdot u_{\sigma,\tau}\cdot\psi_\sigma(\hs^\sigma\!w_\tau)^{-1}\cdot w_\sigma^{-1}\,,
                                                        \quad\psi'_\sigma=\inn(w_\sigma)\circ\psi_\sigma\,.
\end{align*}
We obtain
\[\psi''_\sigma(a)=b\hs\psi_\sigma(\hs^\sigma\!b^{-1} a\hs^\sigma\!b)\hs b^{-1}=
b\hs\psi_\sigma(\hs^\sigma\!b)^{-1}\psi_\sigma(a)\psi_\sigma(\hs^\sigma\!b)\hs b^{-1}=w_\sigma\hs\psi_\sigma(a)\hs w_\sigma^{-1}\,,\]
and so
\[\psi''_\sigma=\inn(w_\sigma)\circ\psi_\sigma=\psi'_\sigma\,,\]
as required.

We have
\begin{align*}
 u'_{\sigma,\tau}= &w_{\sigma\tau}\cdot u_{\sigma,\tau}\cdot\psi_\sigma(\hs^\sigma\!w_\tau)^{-1}\cdot w_\sigma^{-1}\\
=&b\cdot\psi_{\sigma\tau}(\hs^{\sigma\tau}\!b)^{-1}\cdot u_{\sigma,\tau}\cdot
\psi_\sigma(\hs^\sigma\! b\, ^\sigma\!\psi_\tau(\hs^\tau\!b)^{-1})^{-1}\cdot\psi_\sigma(\hs^\sigma\!b)\cdot b^{-1}.
\end{align*}
Since $(u,\psi)\in Z^2(\G,A\to \Inn A)$, we have
\[\psi_{\sigma\tau}(\hs^{\sigma\tau}\!b)=u_{\sigma,\tau}\cdot\psi_\sigma(\hs^\sigma\!\psi_\tau(\hs^\tau\!b))\cdot u_{\sigma,\tau}^{-1}\,.\]
We obtain that
\begin{align*}
u'_{\sigma,\tau}
&=b\cdot u_{\sigma,\tau}\cdot\psi_\sigma(\upsig \psi_\tau(\uptau b))^{-1}\cdot u_{\sigma,\tau}^{-1}\cdot u_{\sigma,\tau}\cdot
\psi_\sigma(\upsig \psi_\tau(\uptau b))\cdot\psi_\sigma(\upsig b)^{-1}\cdot\psi_\sigma(\upsig b)\cdot b^{-1}\\
&=b\cdot u_{\sigma,\tau}\cdot b^{-1}=u''_{\sigma,\tau}\,.
\end{align*}
Thus $u'_{\sigma,\tau}=u''_{\sigma,\tau}$ as required, which completes the proof.

\section{Proof of Theorem \ref{t:DD}}
\label{s:C}

Let
\begin{equation}\label{e:ABC-bis}
1\to A\labelt{i} B\labelt{j} C\to 1
\end{equation}
be a short exact sequence of $\G$-groups.
We construct a connecting map
\[\Delta\colon H^1(C)\to H^2(A\to\Inn B)\quad \text{(sic!).} \]

Let $c\in Z^1(\G,C)\subset\Maps(\G,C)$. We lift $c$ to some map $b\colon \G\to B$ and define
\begin{align*}
&u_{\sigma,\tau}=b_{\sigma\tau}\cdot\hs^\sigma\hlb b_\tau^{-1}\cdot b_\sigma^{-1}\in A\\
&\psi_\sigma=\inn(b_\sigma)\in\Inn B
\end{align*}
for $\sigma,\tau\in \G$.
We set $\Delta([c])=[u,\psi]$, where $[c]$ denotes the class of $c$ in $H^1(C)$
and $[u,\psi]$ denotes the class of $(u,\psi)\in Z^2(\G,A\to\Inn B)$ in $H^2(A\to\Inn B)$.

We check that $(u,\psi)\in Z^2(\G,A\to\Inn B)$.
We have
\begin{align*}
 u_{\sigma,\tau\ups}\cdot\psi_\sigma(\hs^\sigma\hl u_{\tau,\ups})&=
b_{\sigma\tau\ups}\cdot \hs^\sigma\hlb b_{\tau\ups}^{-1}\cdot b_\sigma^{-1}\cdot b_\sigma\cdot
\hs^\sigma\hlb ( b_{\tau\ups}\cdot\hs^\tau\hlb b_\ups^{-1}\cdot b_\tau^{-1})\cdot b_\sigma^{-1}\\
&=b_{\sigma\tau\ups}\cdot\hs^{\sigma\tau}\hlb b_\ups^{-1}\cdot\hs^\sigma\hlb b_\tau^{-1}\cdot b_\sigma^{-1}\\
&=(b_{\sigma\tau\ups}\cdot\hs^{\sigma\tau}\hlb b_\ups^{-1}\cdot b_{\sigma\tau}^{-1})\cdot
(b_{\sigma\tau}\cdot\hs^\sigma\hlb b_\tau^{-1}\cdot b_\sigma^{-1})\\
&=u_{\sigma\tau,\ups}\cdot u_{\sigma,\tau}\,.
\end{align*}
We have
\begin{align*}
\rho(u_{\sigma,\tau})\cdot\psi_\sigma\cdot\hs^\sigma\hl\psi_\tau
&= \inn(b_{\sigma\tau}\cdot\hs^\sigma\hlb b_\tau^{-1}\cdot b_\sigma^{-1})\cdot\inn(b_\sigma)\cdot\hs^\sigma\inn(b_\tau)=\inn(b_{\sigma\tau})=\psi_{\sigma\tau}\,.
\end{align*}
Thus indeed  $(u,\psi)\in Z^2(\G,A\to\Inn B)$, and therefore $[u,\psi] \in H^2(A\to\Inn B)$ makes sense.

We check that the map $\Delta$ is well defined.
First, let $b'=ab$ with $b'_\sigma=a_\sigma\hs b_\sigma$  be another lift of $c$ (here $a_\sigma\in A$).
Then
\begin{align*}
u'_{\sigma,\tau}&=a_{\sigma\tau}\cdot b_{\sigma\tau}\cdot\hs^\sigma\hlb b_\tau^{-1}\cdot
\hs^\sigma\hl a_\tau^{-1}\cdot b_\sigma^{-1}\cdot a_\sigma^{-1}\\
&=a_{\sigma\tau}\cdot b_{\sigma\tau}\cdot\hs^\sigma\hlb b_\tau^{-1}\cdot b_\sigma^{-1}\cdot
                                    b_\sigma\cdot\hs^\sigma\hl a_\tau^{-1}\cdot b_\sigma^{-1}\cdot a_\sigma^{-1}\\
&=a_{\sigma\tau}\cdot u_{\sigma\tau}\cdot\psi_\sigma(\hs^\sigma\hl a_\tau)^{-1}\cdot a_\sigma^{-1}\,,
\end{align*}
Further,
\[\psi'_\sigma=\inn(a_\sigma\, b_\sigma)=\inn(a_\sigma)\cdot \psi_\sigma.\]
We see that $(u',\psi')=a*(u,\psi)\sim(u,\psi)$.

Now let us take another cocycle $c''\sim c$, i.e.,
\[c''_\sigma=\gamma\cdot c_\sigma\cdot\hs^\sigma\hl \gamma^{-1}\]
 for some $\gamma\in C$.
We lift $\gamma$ to some $g\in B$, and we lift $c''$ to $b''\colon\G\to B$ with
\[b''_\sigma=g\cdot b_\sigma\cdot\hs^\sigma \hl g^{-1}.\]
Then
\begin{align*}
u''_{\sigma,\tau}&=g\cdot b_{\sigma\tau}\cdot\hs^{\sigma\tau}\hl g^{-1}\cdot
\hs^\sigma\hlb(g\cdot b_\tau\cdot\hs^\tau\hl g^{-1})^{-1}\cdot
                      ( g\cdot b_\sigma\cdot\hs^\sigma\hl g^{-1} )^{-1}\\
& = g\cdot u_{\sigma,\tau} \cdot g^{-1}=\hs^{\inn(g)} u_{\sigma,\tau}\,.
\end{align*}
Further,
\[\psi''_\sigma=\inn(g\cdot b_\sigma\cdot\hs^\sigma\hl g^{-1})=
\inn(g)\cdot\psi_\sigma\cdot\hs^\sigma\hlb\inn(g)^{-1}\]
and  $(u'',\psi'')=\inn(g)\star(u,\psi)\sim(u,\psi)$. Thus  the map $\Delta$ is indeed well defined.

Consider the morphisms of crossed modules
\[\xymatrix@C=19pt{ (A\to\Inn B)\ar[r]^{i_*} &(B\to\Inn B)\ar[r]^{j_*} &(C\to \Inn C)  }\]
and the sequence
\begin{equation}\label{e:DD-bis}
\xymatrix@C=19pt{
H^1(B)\ar[r]^{j_*} &H^1(C)\ar[r]^-{\Delta} &H^2(A\to\Inn B)\ar[r]^{i_*} &H^2(B\to\Inn B)\ar[r]^{j_*} &H^2(C\to \Inn C)
}
\end{equation}

\begin{theorem}[Dedecker {\cite[Thm.\,2.2]{Dedecker69}} and Debremaeker {\cite[Ch.\,IV, Thm.\,2.1.7, p.\,103]{Debremaeker-thesis}}]
\label{t:DD-bis}
For an exact sequence of $\G$-groups \eqref{e:ABC-bis}, the sequence  \eqref{e:DD-bis} is exact in the following sense:
\begin{enumerate}
\item[(i)] an element of $H^1(C)$ is contained in the image of $H^1(B)$
if and only if its image in $H^2(A\to\Inn B)$ is {\em neutral};
\item[(ii)] an element of $H^2(A\to \Inn B)$ is contained in the image  of $H^1(C)$
if and only if its image in $H^2(B\to\Inn B)$ is {\em the unit element};
\item[(iii)] an element of $H^2(B\to\Inn B)$ is contained in the image of $H^2(A\to\Inn B)$
if and only if its image in $H^2(C\to\Inn C)$ is {\em neutral}.
\end{enumerate}
\end{theorem}

\begin{proof}
Let $b\in Z^1(\G,B)$. We show that $\Delta\circ j_*$ takes $[b]$ to a neutral class.
Indeed, by the definition of $\Delta$, the composite map $\Delta\circ j_*$ maps $[b]$ to
$[u^A,\psi^A]$ where
\[   u^A_{\sigma,\tau}=b_{\sigma\tau}\cdot\hs^\sigma\hlb b_\tau^{-1}\cdot b_\sigma^{-1}=1  \]
because $b$ is a cocycle. Thus $[u^A,\psi^A]=[1,\psi^A]$ is a neutral class.

Conversely, let $c\in Z^1(\G,C)$ and assume that $\Delta$ takes $[c]$ to a neutral class.
Let us lift $c$ to some map $b\colon \G\to B$.
Then $\Delta[c]=[u^A,\psi^A]$, where
\begin{align*}
& u^A_{\sigma,\tau}=b_{\sigma\tau}\cdot\hs^\sigma \hlb b_\tau^{-1}\cdot b_\sigma^{-1},\\
&\psi^A_\sigma=\inn(b_\sigma).
\end{align*}
By assumption $[u^A,\psi^A]$ is a neutral class in $H^2(A\to\Inn B)$, i.e., there exists a map $a\colon \G\to A$ such that
\[ a*(u^A,\psi^A)=(1,\psi^{\prime A}).\]
This means that
\[ a_{\sigma\tau}\cdot b_{\sigma\tau}\cdot\hs^\sigma \hlb b_\tau^{-1}\cdot b_\sigma^{-1}
\cdot b_\sigma\cdot \hs^\sigma\hl a_\tau^{-1}\cdot b_\sigma^{-1}\cdot a_\sigma^{-1}=1.\]
that is,
\[ a_{\sigma\tau}\, b_{\sigma\tau}\cdot\hs^\sigma\hlb(a_\tau\, b_\tau)^{-1}\cdot (a_\sigma\, b_\sigma)^{-1}=1.\]
Set $b'_\sigma=a_\sigma\, b_\sigma$\,, then $b'_{\sigma\tau}=b'_\sigma\cdot\hs^\sigma\hlb b'_\tau$\,, hence $b'$ is a cocycle.
Clearly $j_*$ takes $[b']$ to $[c]$, and therefore, $[c]\in\im j_*$\,, as required.

Let $c\in Z^1(\G,C)$. We show that $i_*\circ\Delta$ takes $[c]$ to 1.
Indeed, let us lift $c$ to some map $b\colon\G\to B$.
Then the composite map  $i_*\circ\Delta$ takes $[c]$ to the class $[u^B,\psi^B]$
where
\begin{align*}
&u^B_{\sigma,\tau}=b_{\sigma\tau}\cdot\hs^\sigma\hlb b_\tau^{-1}\cdot b_\sigma^{-1}\in B\\
&\psi^B_\sigma=\inn(b_\sigma)\in\Inn B,
\end{align*}
and clearly $(u^B,\psi^B)=b*(1,1)$, hence $[u^B,\psi^B]=[1,1]$.

Conversely, let $[u^A,\psi^A]\in H^2(A\to\Inn B)$ and assume that $i_*([u^A,\psi^A])=[1,1]$.
Clearly $i_*([u^A,\psi^A])=[u^A,\psi^A]$, so we obtain that
\[ [u^A,\psi^A]=[1,1]\in H^2(B\to\Inn B). \]
By Theorem \ref{t:Deb} we have $H^2(B\to \Inn B)=\H^2(B\to \Inn B)$, hence $(u^A,\psi^A)=b*(1,1)$ for some $b\colon\G\to B$.
We have
\begin{align*}
&u^A_{\sigma,\tau}=b_{\sigma\tau}\cdot\hs^\sigma\hlb b_\tau^{-1}\cdot b_\sigma^{-1}\,,\\
&\psi^A_\sigma=\inn(b_\sigma).
\end{align*}
Set $c=j\circ b\colon \G\to C$.
Since $u^A_{\sigma,\tau}\in A$, we see that
\[c_{\sigma\tau}\cdot\hs^\sigma\hl c_\tau^{-1}\cdot c_\sigma^{-1}=1, \]
hence $c$ is a cocycle.
Clearly, $[u^A,\psi^A]=\Delta([c])$. Thus $[u^A,\psi^A]\in\im\Delta$, as required.

Let $(u^A,\psi^A)\in Z^2(\G, A\to\Inn B)$.
We show that $j_*\circ i_*$ takes $[u^A,\psi^A]$ to a neutral class.
Indeed, for any $\sigma,\tau\in\G$ we have $u^A_{\sigma,\tau}\in A$.
It follows that the image of $[u^A,\psi^A]$ under the composite map $j_*\circ i_*$
is of the form $[1,\psi^C]$, and hence is neutral.

Conversely, assume that $j_*$ takes $[u^B,\psi^B]$ to a neutral class $[u^C,\psi^C]$.
This means that there exists a map $c\colon \G\to C$ such that $c*(u^C,\psi^C)=(1,\psi^{\prime C})$.
Let us lift $c$ to a map $b\colon \G\to B$ and set $(u^{\prime B},\psi^{\prime B})=b*(u^B,\psi^B)$.
Then for any $\sigma,\tau\in G$ we have $u^{\prime B}_{\sigma,\tau}\in A$.
We see that  $[u^B,\psi^B]=[u^{\prime B},\psi^{\prime B}]$ lies in the image of $i_*$\,, as required.
This completes the proof of the theorem.
\end{proof}

\end{document}